\documentclass{article}
\usepackage[utf8]{inputenc}
\usepackage [english]{babel}
\usepackage {geometry}
\usepackage {fancyhdr}
\usepackage {amsmath,amsfonts,amssymb,amsthm,epsfig,epstopdf,titling,url,array}
\usepackage {mathrsfs}
\usepackage {graphicx}
\usepackage {hyperref}
\usepackage {enumerate}
\usepackage {wrapfig}
\usepackage {cleveref}

\theoremstyle{plain}
\newtheorem{thm}{Theorem}[section]
\newtheorem{lem}[thm]{Lemma}
\newtheorem{prop}[thm]{Proposition}
\newtheorem{cor}[thm]{Corollary}

\theoremstyle{definition}
\newtheorem{defn}[thm]{Definition}

\theoremstyle{remark}

\newcommand{\Z}{\mathbb{Z}}
\newcommand{\Q}{\mathbb{Q}}

\newcommand{\C}{\mathbb{C}}

\newcommand{\OO}{\mathcal{O}}

\newcommand{\X}{\bar{X}}

\newcommand{\e}{\hat{\varepsilon}}
\newcommand{\dn}{\text{dn}}
\newcommand{\beg}{\text{beg}}
\newcommand{\eend}{\text{end}}

\begin{document}

\renewcommand{\abstractname}{\vspace{-\baselineskip}}

\title{An explicit determination of the $K$-theoretic structure constants of the affine Grassmannian associated to $SL_2$}

\author{Seth Baldwin}

\maketitle

\begin{abstract}
Let $G:=\widehat{SL_2}$ denote the affine Kac-Moody group associated to $SL_2$ and $\bar{\mathcal{X}}$ the associated affine Grassmannian.  We determine an inductive formula for the Schubert basis structure constants in the torus-equivariant Grothendieck group of $\bar{\mathcal{X}}$.  In the case of ordinary (non-equivariant) $K$-theory we find an explicit closed form for the structure constants.  We also determine an inductive formula for the structure constants in the torus-equivariant cohomology ring, and use this formula to find closed forms for some of the structure constants.  
\end{abstract}

\section{Introduction}

Let $G:=\widehat{SL_2}$ be the affine Kac-Moody group associated to $SL_2$, completed along the negative roots.  Let $P$ be the standard maximal parabolic subgroup and $\bar{\mathcal{X}}:=G/P$ the thick affine Grassmannian.  Let $H$ denote the standard maximal torus, and let $T:=H/Z(G)$ where $Z(G)$ denotes the center of $G$.  Then the natural left action of $H$ on $\bar{\mathcal{X}}$ descends to an action of $T$ on $\bar{\mathcal{X}}$.  Let $R(T)$ denote the representation ring of $T$, and let $W$ denote the affine Weyl group.

\subsection{Equivariant $K$-theory} Let $K^0_T(\bar{\mathcal{X}})$ denote the Grothendieck group of $T$-equivariant coherent sheaves on $\bar{\mathcal{X}}$.  Then the structure sheaves of the opposite (finite codimension) Schubert varieties in $\bar{\mathcal{X}}$ form a `basis' (where infinite sums are allowed) of $K^0_T(\bar{\mathcal{X}})$.  More precisely, we have $$K^0_T(\bar{\mathcal{X}})=\prod_{k\in \Z_{\geq 0}} R(T) \hat{\OO}^k,$$
where $\hat{\OO}^k$ denotes the class of the structure sheaf of the (unique) opposite Schubert variety of codimension $k$.
The structure constants $d^k_{n,m}\in R(T)$ are defined by 

\begin{equation}
\label{strucSL2intro}\hat{\OO}^n\cdot\hat{\OO}^m=\sum_{k\geq n,m} d^k_{n,m}\hat{\OO}^k.
\end{equation}

Using a Chevalley formula due to Lenart-Shimozono \cite[Corollary 3.7]{LeSh}, phrased in the Lakshmibai-Seshadri path model, we explicitly compute the structure constants $d^k_{1,m}$ corresponding to multiplication by the Schubert divisor $\hat{\OO}^1$ (see (\ref{dk1m})).  Next, the structure constants $d^m_{n,m}$ (where $n\leq m$) are computed (see (\ref{dmnm})) using a result of Lam-Schilling-Shimozono on the localizations of Schubert varieties \cite[Proposition 2.10]{LSS}.  Then, using the associative law in the $K$-group, we derive an inductive formula (Proposition \ref{inductprop}) for the structure constants using $d^m_{n,m}$ and $d^k_{1,m}$ as our base cases.

Let $\hat{\xi}^k$ denote the ideal sheaf `basis' (where we allow infinite sums) dual to the basis $\hat{\OO}^k$ (see \S 4.1).  We define the structure constants $b^k_{n,m}\in R(T)$ in the basis $\hat{\xi}^k$ by 
\begin{equation}
\label{strucSL2introxi}\hat{\xi}^n\cdot\hat{\xi}^m=\sum_{k\geq n,m} b^k_{n,m}\hat{\xi}^k.
\end{equation}
Then we have (see (\ref{bd})) $$b^k_{n,m}=\sum_{j\leq k} \left(d^j_{n,m}-d^j_{n+1,m}-d^j_{n,m+1}+d^j_{n+1,m+1}\right).$$  Thus, in principle, the $b^k_{n,m}$ can be computed from the $d^k_{n,m}$.  

\subsection{Ordinary $K$-theory} Let $K^0(\bar{\mathcal{X}})$ denote the Grothendieck group of coherent sheaves on $\bar{\mathcal{X}}$.  Let $\OO^k$ denote the class of the structure sheaf of the (unique) opposite Schubert variety of codimension $k$.  Then, we have $$K^0(\bar{\mathcal{X}})=\prod_{k\in \Z_{\geq 0}}\Z \OO^k.$$  

Further, the structure constants in $K^0(\bar{\mathcal{X}})$ are given by evaluating the $T$-equivariant structure constants at $1$.  Thus we denote the structure constants in the basis $\OO^k$ in ordinary $K$-theory by $d^k_{n,m}(1)$, so that we have, in $K^0(\bar{\mathcal{X}})$, $$\OO^n\cdot\OO^m=\sum_{k\geq n+m} d^k_{n,m}(1)\OO^k.$$
Similarly, letting $\xi^k$ denote the ideal sheaf `basis' (where we allow infinite sums) dual to the basis $\OO^k$, we denote the structure constants in the basis $\xi^k$ in ordinary $K$-theory by $b^k_{n,m}(1)$, so that we have, in $K^0(\bar{\mathcal{X}})$, $$\xi^n\cdot\xi^m=\sum_{k\geq n+m} b^k_{n,m}(1)\xi^k.$$

Then the following theorem, which we prove using our inductive formula, gives a closed form for the structure constants in ordinary $K$-theory (see Theorem \ref{dknmordthm} and Corollary \ref{bcor}):

\begin{thm}
The structure constants in ordinary $K$-theory are given by 
$$d^{n+m+k}_{n,m}(1)=(-1)^k\cdot\frac{(n+m+k-1)!}{(n-1)!(m-1)!k!}\cdot\frac{n+m+2k}{(n+k)(m+k)}.$$
$$b^{n+m+k}_{n,m}(1)=(-1)^k\cdot\frac{(n+m+k)!}{n!m!k!}.$$
\end{thm}

\subsection{Equivariant cohomology} Let $G^{\min}$ denote the minimal affine Kac-Moody group associated to $SL_2$.  Let $P$ be the standard maximal parabolic subgroup and let $\mathcal{X}:=G^{\min}/P$ denote the standard affine Grassmannian.  

Then the $T$-equivariant cohomology ring, $H_T^\bullet(\mathcal{X})$, has a Schubert basis (see \cite[Theorem 11.3.9]{Kbook}), which we denote by $\{\e_i\}_{i=0}^\infty$. Let $\mathbb{Z}[\alpha_0,\alpha_1]$ denote the graded ring of polynomials with integral coefficients in the simple roots $\alpha_0$ and $\alpha_1$.  Further, let $\mathbb{Z}[\alpha_0,\alpha_1](k)$ denote the $k$-th graded piece of $\mathbb{Z}[\alpha_0,\alpha_1]$.  

We define the $T$-equivariant cohomology structure constants $c^k_{n,m}\in \mathbb{Z}[\alpha_0,\alpha_1](n+m-k)$ (see \cite[Corollary 11.3.17]{Kbook}) by
\begin{equation}\label{cknmintro}
\e_n\cdot\e_m=\sum_{k=\max\{n,m\}}^{n+m} c^k_{n,m}\e_k.
\end{equation}
We derive an inductive formula for the structure constants in $T$-equivarient cohomology (see Proposition \ref{scindprop}).  Using this formula, we derive closed forms for $c^{n+m}_{n,m}$, $c^{n+m-1}_{n,m}$, and $c^{n+m-2}_{n,m}$.  The structure constants $c^m_{n,m}$ (where $n\leq m$) are determined by \cite[Lemma 11.1.10 and Proposition 11.1.11 (1) and (3)]{Kbook}.

\subsection{Summary of paper} What follows is a brief summary of the rest of the paper.  In section $2$ we introduce general notation for Kac-Moody groups and their flag varieties.  In section $3$ we summarize the relevant generalities about $K$-theory.  We also state a result due to Lam-Schilling-Shimozono \cite{LSS} on the localizations of Schubert varieties.  Further, we introduce the notion of Lakshmibai-Seshadri path which allows us to state a Chevalley formula due to Lenart-Shimozono \cite{LeSh}.  In section $4$ we specialize to the case of affine $SL_2$ and determine an explicit closed form for the Chevalley coefficients.  Then, we derive an inductive formula for the structure constants in $T$-equivariant $K$-theory.  In section $5$, we use our inductive formula to determine a closed form for the structure constants in ordinary $K$-theory.  Finally in section $6$ we move to the case of $T$-equivariant cohomology, where we derive an inductive formula for the structure constants and determine closed forms for $c^{n+m}_{n,m}$, $c^{n+m-1}_{n,m}$, $c^{n+m-2}_{n,m}$ and when $n\leq m$, for $c^m_{n,m}$.\\

\textbf{Acknowledgements.}  The author would like to thank M. Shimozono for providing a Sage package (discussed in \cite[1.5]{LeSh}) which the author used to verify some computations related to the $K$-theoretic Chevalley formula.  

\section{Notation}

We work over the field $\C$ of complex numbers.  Let $G$ be any symmetrizable Kac-Moody group over $\C$ completed along the negative roots (as opposed to completed along the positive roots as in \cite[Chapter 6]{Kbook}).  Further, let $G^{\text{min}}\subset G$ be the minimal Kac-Moody group as in \cite[\S7.4]{Kbook}.  Let $B$ be the standard Borel subgroup, $B^-$ the standard opposite Borel subgroup, $H:=B\cap B^-$ the standard maximal torus, $T:=H/Z(G^{\min})$ the adjoint torus, where $Z(G^{\min})$ is the center of $G^{\min}$.  Let $W$ denote the Weyl group.  Let $\bar{X}:=G/B$ denote the thick flag variety (introduced by Kashiwara \cite{Ka}), which contains the standard flag variety $X=G^{\text{min}}/B$.  When $G$ is infinite dimensional, $\bar{X}$ is an infinite dimensional non-quasi compact scheme, whereas $X$ is an ind-projective variety \cite[\S7.1]{Kbook}.  The natural left actions of $H$ on $\X$ and $X$ descend to actions of $T$ on $\X$ and $X$.  

For any $w\in W$ we have the Schubert cell $$C_w:=BwB/B\subset X ,$$ the Schubert variety $$X_w:=\overline{C_w}=\bigsqcup_{w'\leq w} C_{w'}\subset X ,$$ the opposite Schubert cell $$C^w:=B^{-}wB/B\subset \X ,$$ and the opposite Schubert variety $$X^w :=\overline{C^w}=\bigsqcup_{w'\geq w} C^{w'}\subset\X ,$$ all endowed with the reduced subscheme structures.  Then, $X_w$ is a (finite dimensional) irreducible projective subvariety of $X$ and $X^w$ is a finite codimensional irreducible subscheme of $\X$ \cite[\S7.1]{Kbook} and \cite[\S4]{Ka}.

Let $R(T)$ denote the representation ring of $T$.  For any integral weight $\lambda$ let $\C_\lambda$ denote the one-dimensional representation of $T$ on $\C$ given by $t\cdot v=\lambda(t)v$ for $t\in T, v\in\C$.  By extending this action to $B$ we may define, for any integral weight $\lambda$, the $G$-equivariant line bundle $\mathcal{L}(\lambda)$ on $\X$ by $$\mathcal{L}(\lambda):=G\times^B \C_{-\lambda},$$ where for any representation $V$ of $B$, $G\times^B V:=(G\times V)/B$ where $B$ acts on $G\times V$ via $b(g,v)=(g b^{-1}, bv)$ for $g\in G, v\in V, b\in B$.  Then $G\times^B V$ is the total space of a $G$-equivariant vector bundle over $X$, with projection given by $(g,v)B\mapsto gB$.  

\section{Grothendieck group and Chevalley formula}

In this section we introduce the $T$-equivariant Grothendieck group of $\X$.  We then explain the relationship between the Grothendieck groups of complete and partial flag varieties, and the relationship between the $T$-equivariant and ordinary Grothendieck groups.  Next, we state a closed form for the localizations in the Schubert basis due to Lam-Schilling-Shimozono \cite{LSS}.  Finally, we introduce the concept of Lakshmibai-Seshadri path, which allows us to state a Chevalley formula due to Lenart-Shimozono \cite{LeSh}.  

\subsection{Grothendieck group}

Let $K^0_T(\X)$ denote the Grothendieck group of $T$-equivariant coherent $\OO_{\X}$-modules.  For any $u\in W$, $\OO_{X^u}$ is a coherent $\OO_{\X}$-module, by \cite[\S2]{KS}.  Since $K^0_T(pt)=R(T)$, we have that $K^0_T(\X)$ is a module over $R(T)$.  Further, from \cite[comment after Remark 2.4]{KS} we have:

\begin{prop}
\label{basis}  
$\{[\OO_{X^u}]\}$ forms a `basis' of $K^0_T(\X)$ as an $R(T)$-module (where we allow infinite sums), i.e., $$K^0_T(\bar{X})=\prod_{w\in W} R(T) [\OO_{X^w}].$$
\end{prop}

The structure constants $d^w_{u,v}\in R(T)$ are defined by 
\begin{equation}
\label{struc}[\OO_{X^u}]\cdot[\OO_{X^v}]=\sum_{w\in W} d^w_{u,v}[\OO_{X^w}].
\end{equation}
Note that for fixed $u,v\in W$, infinitely many of the $d^w_{u,v}$ may be nonzero.  We also have $d^w_{u,v}=0$ unless $w\geq u,v$.  

\subsection{Relation between structure constants for $G/B$ and $G/P$}

Now letting $P$ be any standard parabolic subgroup, we have $$K^0_T(G/P)=\prod_{w\in W^P} R(T) [\OO_{X_P^w}],$$ where $W^P$ denotes the set of minimal length representatives of $W/W_P$, and $X_P^w:=\overline{BwP/P}$.  

We define the structure constants $d^w_{u,v}(P)$ for $G/P$ in the analogous way.

$$[\OO_{X_P^u}]\cdot[\OO_{X_P^v}]=\sum_{w\in W^P} d^w_{u,v}(P)[\OO_{X_P^w}].$$

Let $\pi:G/B\to G/P$ be the standard ($T$-equivariant) projection.  Then, $\pi$ is a locally trivial fibration (with fiber the smooth projective variety $P/B$) and hence flat (see \cite[Chapter 7]{Kbook}).  Thus, we have $$\pi^*[\OO_{X^w_P}]=[\OO_{\pi^{-1}(X^w_P)}]=[\OO_{X^w}].$$  Since $\pi^*:K^0_T(G/P)\to K^0_T(G/B)$ is a ring homomorphism, we have $$d_{u,v}^w=d_{u,v}^w(P)$$ for any $u,v,w\in W^P$.  Thus we henceforth drop the notation $d^w_{u,v}(P)$ in favor of the notation $d^w_{u,v}$, even when working with partial flag varieties.  

\subsection{Relation between structure constants for ordinary and equivariant $K$-theory}

Let $K^0(\X)$ denote the Grothendieck group of coherent sheaves on $\X$.  Then, we have $$K^0(\X)=\prod_{w\in W}\Z \overline{[\OO_{X^w}]},$$ where $\overline{[\OO_{X^w}]}$ denotes the class of $\OO_{X^w}$ in $K^0(\X)$.  Further, the map $$\Z\otimes_{R(T)}K^0_T(\X)\to K^0(\X),\text{ }1\otimes [\OO_{X^w}]\mapsto \overline{[\OO_{X^w}]}$$ is an isomorphism, where we view $\Z$ as an $R(T)$-module via evaluation at $1$.  Similar results apply to $G/P$.  Hence, we have, in $K^0(\X)$, $$\overline{[\OO_{X^u}]}\cdot\overline{[\OO_{X^v}]}=\sum_{w\in W} d^w_{u,v}(1)\overline{[\OO_{X^w}]},$$ and similarly for any parabolic subgroup $P$, we have, in $K^0(G/P)$, 
\begin{equation}\label{ev1}
\overline{[\OO_{X_P^u}]}\cdot\overline{[\OO_{X_P^v}]}=\sum_{w\in W^P} d^w_{u,v}(1)\overline{[\OO_{X_P^w}]}.
\end{equation}
We also have $d^w_{u,v}(1)=0$ unless $w\geq u+v$.  

\subsection{Localizations in the Schubert basis}

We identify the set $\{ wB\}$ of $T$-fixed points of $\bar{X}$ with the Weyl group $W$.  Given $x\in W$, let $i_x: \{x\}\hookrightarrow X$ denote the inclusion map.  Then pullback induces a ring homomorphism $i_x^*:K^0_T(\bar{X})\to K^0_T(\{x\})\cong R(T)$.  For $\psi\in K^0_T(\bar{X})$ and $x\in W$, the localization of $\psi$ at $x$ is defined as $$\psi_x:=i_x^*(\psi).$$  

We are concerned with localizations in the basis $[\OO_{X^w}]$.  The following is \cite[Lemma 2.3]{LSS}:

\begin{lem}[\cite{LSS} Lemma 2.3]
$[\OO_{X^w}]_x=0$ unless $x\geq w$.  
\end{lem}

\begin{defn}\label{defn*}
For $u,v\in W$, the set $\{xy:x\leq u, y\leq v \}$ has a maximum element, which we denote by $u*v$ (see \cite[Lemma 1.4]{H}).  Then $s_i*s_j=s_is_j$ if $i\neq j$, while $s_i*s_i=s_i$.  
\end{defn}

An explicit closed form for the localizations in the Schubert basis is given by \cite[Proposition 2.10]{LSS} (see also \cite[Theorem 3.12]{G} and \cite[]{W} for the same result in the finite case).
\begin{thm}[\cite{LSS} Proposition 2.10]\label{locthm}
Let $x\geq w\in W$.  Fix a reduced decomposition $x=s_{i_1}s_{i_2}\dots s_{i_m}$.  For $\ell\leq m$, define $\beta_\ell:=s_{i_1}s_{i_2}\dots s_{i_{\ell-1}}\alpha_{i_\ell}$.  Then
$$[\OO_{X^w}]_x=\sum(-1)^{\ell(w)}(e^{\beta_{j_1}}-1)\dots(e^{\beta_{j_p}}-1),$$
where the summation runs over all $1\leq j_1<\dots<j_p\leq m$ such that $w=s_{i_{j_1}}*\dots*s_{i_{j_p}}$.
\end{thm}

Localizing the equation defining the structure constants (\ref{struc}) at $x$ gives $$[\OO_{X^u}]_x\cdot[\OO_{X^v}]_x=\sum_{w\in W} d^w_{u,v}[\OO_{X^w}]_x.$$  Now if $u\leq v$, then since $[\OO_{X^w}]_x=0$ unless $x\geq w$, and $d^w_{u,v}=0$ unless $w\geq u$ and $w\geq v$, letting $x=v$ gives $[\OO_{X^u}]_v\cdot[\OO_{X^v}]_v=d^v_{u,v}[\OO_{X^v}]_v,$ which reduces to 
\begin{equation}\label{locdvuv}
[\OO_{X^u}]_v=d^v_{u,v},\,\,\,\text{ for }u\leq v.
\end{equation}

\subsection{Lakshmibai-Seshadri paths}

In this subsection we introduce the notion of Lakshmibai-Seshadri paths.  We do not attempt to give this subject a proper treatment, but instead introduce only the notions necessary to understand the statement of the Chevalley formula in the subsequent subsection.  

Let $S=\{s_i:i\in I\}$ denote the set of simple reflections of $W$.  For any $J\subset I$, define $W_J$ to be the Weyl group generated by the $s_j$ where $j\in J$.  Let $\lambda$ be a dominant integral weight.  Then its stabilizer $W_\lambda$ is the parabolic subgroup $W_J$ with $J = \{i \in I : s_i \lambda = \lambda\}$.

We define the Bruhat ordering on the orbit $W\lambda$ of $\lambda$ by taking the transitive closure of the relations $s_\alpha \sigma < \sigma$ iff $\langle \sigma, \alpha^\vee\rangle>0$, where $\alpha$ is a positive root and $\sigma\in W\lambda$.  Note that by this convention, for $u,v\in W^\lambda$, we have $u < v$ iff $v\lambda < u \lambda$ (where $W^\lambda$ denotes the set of minimal length representatives on $W/W_\lambda$).  For a real number $b$, we define the $b$-Bruhat ordering `$<_b$' on $W\lambda$ by defining $\mu$ to cover $\nu$ in the $b$-Bruhat order iff $\mu$ covers $\nu$ in the normal Bruhat order and $b(\mu-\nu)$ is an integer multiple of a root.   

\begin{defn}[Lakshmibai-Seshadri path \cite{St}]\label{LSpathdef}
A Lakshmibai-Seshadri (LS) path $p$ of shape $\lambda$ is a pair $p=(\sigma, b)$ where $$\sigma: \sigma_1 > \sigma_2 >\dots > \sigma_m,\text{ }\sigma_i\in W/W_\lambda$$ $$b: 0=b_0<b_1<\dots<b_m=1,\text{ }b_i\in\Q.$$  We also require that $$\sigma_1 \lambda <_{b_1} \sigma_2 \lambda <_{b_2} \dots <_{b_{m-1}} \sigma_m \lambda.$$  
\end{defn}

Denote by $\mathcal{T}^\lambda$ the set of all LS paths of shape $\lambda$.  For $p\in \mathcal{T}^\lambda$ we define the weight of $p=(\sigma, b)$ to be 
\begin{equation}\label{p(1)}
p(1)=\sigma_m \lambda - \sum_{i=1}^{m-1}b_i(\sigma_{i+1} \lambda - \sigma_i\lambda).
\end{equation} 

\begin{prop}[\cite{LS2}, Lemma 4.4']\label{propdn}
Let $\tau\in W/W_J$ and $w\in W$ be such that $wW_J \geq \tau$ in $W/W_J$.  Then the set
$$\{v\in W: w\geq v, vW_J=\tau\}$$ has a Bruhat-maximum, which will be denoted by $\dn(w,\tau)$. The symbol $\dn$ is an abbreviation for ``down".
\end{prop}

For $p\in\mathcal{T}^\lambda$, we define, with notation as in Definition \ref{LSpathdef},  $$\beg(p)=\sigma_1\text{ and } \eend(p)=\sigma_m.$$  
Then for $w \in W$ such that $\beg(p) \leq wW_\lambda$, define $\dn(w,p)$ by 
\begin{equation}\label{downdef}
w=w_0\geq w_1\geq \dots\geq w_m=\dn(w,p),
\end{equation}
where $w_i:=\dn(w_{i-1},\sigma_i)$ for $i$ from $1$ to $m$.  Here $\dn(w_{i-1},\sigma_i)$ is defined as in Proposition \ref{propdn}.  

For $z,w\in W$ we define 
\begin{equation}\label{D}
\mathcal{D}^\lambda_{w,z}:=\{p\in \mathcal{T}^\lambda:\beg(p)\leq wW_\lambda,\dn(w,p)=z\}.
\end{equation}

\subsection{Chevalley formula}

For any integral weight $\lambda$, define the Chevalley coefficients $a^w_v(\lambda)\in R(T)$ by 
\begin{equation}
\label{cheva}[\mathcal{L}(\lambda)]\cdot[\OO_{X^v}]=\sum_{w\geq v}a^w_v(\lambda)[\OO_{X^w}].
\end{equation}
Now letting $\lambda$ be a dominant integral weight, we have the following Chevalley formula from \cite{LeSh} (note that $\mathcal{L}(\lambda)$ in our notation is written $\mathcal{L}^{-\lambda}$ in the notation of \cite{LeSh}):
\begin{thm}[\cite{LeSh}, Corollary 3.7]
For any dominant integral weight $\lambda$, we have
$$[\mathcal{L}(\lambda)]\cdot[\OO_{X^v}]=\sum_{w\geq v}\sum_{p\in \mathcal{D}^{\lambda}_{w,v}}(-1)^{\ell(w)-\ell(v)} e^{-p(1)}[\OO_{X^w}],$$
where $\mathcal{D}^{\lambda}_{w,z}$ is defined in (\ref{D}) and $p(1)$ is defined in (\ref{p(1)}).  
\end{thm}

This immediately gives
\begin{equation}\label{a}
a^w_v(\lambda)=\sum_{p\in \mathcal{D}^{\lambda}_{w,v}}(-1)^{\ell(w)-\ell(v)} e^{-p(1)}.
\end{equation}
Further, for any simple reflection $s_i$, we have $[\OO_{X^{s_i}}]=1-e^{\Lambda_i}[\mathcal{L}(\Lambda_i)]$, where $\Lambda_i$ denotes the $i$-th fundamental weight.  Hence,
\begin{equation}\label{da}
d^w_{s_i,v}= 
\begin{cases} 
      -e^{\Lambda_i}a^w_v(\Lambda_i) & w\neq v \\
      1-e^{\Lambda_i}a^v_v(\Lambda_i) & w=v.
\end{cases}
\end{equation}

\section{Structure constants for $\widehat{SL_2}$ in $T$-equivariant $K$-theory}

In this section we specialize to the case of $\widehat{SL_2}$.  We begin by explicitly determine the Chevalley coefficients $a^w_v(\Lambda_0)$ (see (\ref{a})) for the affine Grassmannian associated to $SL_2$, where $\Lambda_0$ is the zeroth fundamental weight.  We then determine an inductive formula for the structure constants.  

\subsection{Notation for $\widehat{SL_2}$}

Let $G:=\widehat{SL_2}$ be the affine Kac-Moody group associated to $SL_2$, completed along the negative roots.  Let $P$ be the standard maximal parabolic subgroup and $\bar{\mathcal{X}}:=G/P$ the thick affine Grassmannian.  Let $T$ denote the standard maximal torus, $R(T)$ its representation ring, and $W$ the affine Weyl group.

Let $\mathfrak{g}:=\widehat{\mathfrak{sl}_2}$.  We denote by $\mathfrak{h}\subset\mathfrak{g}$ the Cartan subalgebra, and by $\mathfrak{h}^*$ its dual.  Then we have the simple roots $\alpha_0, \alpha_1\in\mathfrak{h}^*$, the simple coroots $\alpha_0^\vee,\alpha_1^\vee\in\mathfrak{h}$, the simple reflections $s_0,s_1\in W$, and the fundamental weights $\Lambda_0,\Lambda_1\in\mathfrak{h}^*$.  Note that $W$ is isomorphic to the free product $\langle s_0\rangle *\langle s_1\rangle$.  

Let $W^P$ denote the set of minimal length representatives of $W/W_P$, where $W_P$ denotes the Weyl group of $P$.  Then $W^P=\{w_n\}_{n=0}^\infty$ where $$w_n:=\dots s_0s_1s_0$$ denotes the word of length $n$ in $s_0$ and $s_1$ which is alternating and ends in $s_0$ (so $w_0=e, w_1=s_0, w_2=s_1 s_0, w_3=s_0 s_1 s_0$, etc).  Further, $W/W_P$ is totally ordered under the relative Bruhat ordering.  

We denote by $\hat{\OO}^k:=[\OO_{X_P^{w_k}}]$.  Then we have: $$K^0_T(\bar{\mathcal{X}})=\prod_{k\in \Z_{\geq 0}} R(T) \hat{\OO}^k.$$

We henceforth denote the structure constants $d^{w_k}_{w_n,w_m}$ (see (\ref{struc})) and the Chevalley coefficients $a^{w_k}_{w_m}(\Lambda_0)$ (see (\ref{cheva})) by $d^k_{n,m}$ and $a^k_m$ respectively.  Then, in this notation, we have, in $K^0_T(\bar{\mathcal{X}})$, 
\begin{equation}
\label{strucSL2}\hat{\OO}^n\cdot\hat{\OO}^m=\sum_{k\geq n,m} d^k_{n,m}\hat{\OO}^k.
\end{equation}
Further, we have, by (\ref{a}),

\begin{equation}\label{aSL2}
a^k_m=\sum_{p\in \mathcal{D}^{\Lambda_0}_{w_k,w_m}}(-1)^{k+m} e^{-p(1)}.
\end{equation}

We also consider the basis dual to the basis $\hat{\OO}^k$, defined as follows.  Let $$\partial X_P^{w_k}:=X_P^{w_k}\setminus C_P^{w_k}$$ be given the reduced subscheme strcture, where $C_P^{w_k}:=B^-w_kP/P$.  Let $$\hat{\xi}^k:=[\OO_{X_P^{w_k}}(-\partial X_P^{w_k})]$$ denote the class of the ideal sheaf of $\partial X_P^{w_k}$ in $X_P^{w_k}$.  Then $\{\hat{\xi}^k\}_{k=0}^{\infty}$ forms an $R(T)$-`basis' of $K^0(\bar{\mathcal{X}})$, where we allow infinite sums.  Further, we have 
\begin{equation}\label{xiO}
\hat{\xi}^k=\hat{\OO}^k-\hat{\OO}^{k+1}.
\end{equation} 

We define the structure constants $b^k_{n,m}\in R(T)$ in the basis $\hat{\xi}^k$ by 
\begin{equation}\label{scxi}
\hat{\xi}^n\cdot\hat{\xi}^m=\sum_{k\geq n,m}b^k_{n,m}\hat{\xi}^k.
\end{equation}
Using (\ref{xiO}) and looking at the coefficient of $\hat{\OO}^k$ in the product $\hat{\xi}^n\cdot\hat{\xi}^m$ we obtain
$$b^k_{n,m}-b^{k-1}_{n,m}=d^k_{n,m}-d^k_{n+1,m}-d^k_{n,m+1}+d^k_{n+1,m+1}.$$
Thus by induction, we have the following expression for the $b^k_{n,m}$ in terms of the $d^k_{n,m}$ 
\begin{equation}\label{bd}
b^k_{n,m}=\sum_{j\leq k} \left(d^j_{n,m}-d^j_{n+1,m}-d^j_{n,m+1}+d^j_{n+1,m+1}\right).
\end{equation}

\subsection{Determination of $a^k_m$}

We begin by determining the Chevalley coefficients $a^k_m$.  By (\ref{aSL2}), this can be done by determining the set $\mathcal{D}^{\Lambda_0}_{w_k,w_m}$ and the weights of the associated paths.  Once the $a^k_m$ are known, we obtain $d^k_{1,m}$ using (\ref{da}).  

The following is a restatement of \cite[Lemma 1]{S}, although we provide a proof as our conventions are different.

\begin{lem}\label{LSlem}
The LS paths of shape $\Lambda_0$ are those paths $p=(\sigma, b)$ such that $$\sigma : w_{\ell}> w_{\ell-1}> \dots > w_m$$ $$b: 0 < b_{\ell} < b_{\ell-1} < \dots < b_{m+1} < 1,$$ where $\ell \geq m$ and where $b_{j}\cdot j\in \Z$ for all $m+1\leq j\leq \ell$.
\end{lem}

\begin{proof}
It can be easily checked that all such paths are LS paths.  

To complete the proof, we must show that all LS paths are of this form, i.e. we must show that $\sigma$ can have no `skips', and that the $b_j$'s must satisfy the stated condition.

One may compute that 
\begin{equation}
\label{womega}
w_i\Lambda_0=
\begin{cases} 
      \Lambda_0-j^2\alpha_0-(j^2+j)\alpha_1 & i=2j \\
      \Lambda_0-(j+1)^2\alpha_0-(j^2+j)\alpha_1 & i=2j+1.
\end{cases}
\end{equation}
It follows that 
\begin{equation}\label{diff}
w_{i-1}\Lambda_0-w_i\Lambda_0=
\begin{cases}
i\alpha_0 & i\text{ odd}\\
i\alpha_1 & i\text{ even}.
\end{cases}
\end{equation}

Recall that by definition, in the $b$-Bruhat order, $\mu$ covers $\nu$ iff $\mu$ covers $\nu$ in the normal Bruhat order and $b(\mu-\nu)$ is an integer multiple of a root.  By (\ref{diff}), for $0< b < 1$, it is not possible that $b(w_i\Lambda_0-w_{i+1}\Lambda_0)$ and $b(w_{i-1}\Lambda_0-w_i\Lambda_0)$ are both integer multiples of a root.  Hence $\sigma$ can have no skips.  The condition on the $b_j$'s also follows from (\ref{diff}).  

\end{proof}

As noted in \cite{S}, the condition on the $b_j$'s can be rephrased $$b:0< \frac{i_{\ell}}{\ell} < \frac{i_{\ell-1}}{\ell-1} < \dots < \frac{i_{m+1}}{m+1} < 1,$$ where $i_{j}\in \Z_{\geq 0}$ for all $m+1\leq j \leq \ell$.  These equalities are equivalent to the requirement that $1\leq i_{\ell}\leq i_{\ell-1}\leq \dots \leq i_{m+1}\leq m$.  

For $\ell\geq m$, let $p(\ell,m)$ denote the set of all LS paths of shape $\Lambda_0$ beginning at $w_\ell$ and ending at $w_m$:  $$p(\ell,m):=\{p\in \mathcal{T}^{\Lambda_0}: \beg(p)=w_\ell, \eend(p)=w_m\}.$$  

For $p\in p(\ell,m)$, the weight (\ref{p(1)}) becomes: $$p(1)=w_m\Lambda_0-\sum_{j=m+1}^\ell b_j(w_{j-1}\Lambda_0-w_j\Lambda_0).$$
Thus, by (\ref{diff}), we have
\begin{equation}\label{p(1)SL2}
p(1)=w_m \Lambda_0-\sum_{\substack{j\text{ odd}\\ \ell\geq j \geq m+1}}i_j\alpha_0 - \sum_{\substack{j\text{ even} \\ \ell\geq j \geq m+1}}i_j\alpha_1.
\end{equation}

Let $k\geq \ell$.  Note that $W_{\Lambda_0}=W_{s_1}$.  It is easy to see that 
\[
\dn(w_k, w_\ell W_{\Lambda_0})=
\begin{cases}
w_\ell & k=\ell\text{ or }\ell+1\\
w_\ell s_1 & k\geq\ell+2,
\end{cases}
\]
and $$\dn(w_k s_1,w_\ell W_{\Lambda_0})=w_\ell s_1,$$ where $\dn(w_k,w_\ell W_{\Lambda_0})$ is defined by Proposition \ref{propdn}.  It follows that for any path $p\in p(\ell,m)$ we have 
\[
\dn(w_k, p)=
\begin{cases}
 w_m & k=\ell\text{ or }\ell+1\\
w_m s_1& k\geq \ell+2,
\end{cases}
\] 
where by $\dn(w_k, p)$ is defined by (\ref{downdef}).  Thus we have 
\begin{equation}\label{Ddisjoint}
\mathcal{D}^{\Lambda_0}_{w_k,w_m}=p(k,m)\bigsqcup p(k-1,m).
\end{equation}

From (\ref{Ddisjoint}) and (\ref{aSL2}) we have $$a^k_m=(-1)^{k+m}\left[\sum_{p\in p(k,m)} e^{-p(1)}+\sum_{p\in p(k-1,m)} e^{-p(1)}\right].$$
Thus by (\ref{da}), (\ref{p(1)SL2}), and (\ref{womega}), we obtain
\begin{equation}\label{dk1m}
d^k_{1,m}=
\begin{cases} 
      \displaystyle (-1)^{k+m+1}e^{q_m}\left[\sum_{\substack{\mathbf{i}=(i_k,\dots,i_{m+1})\in\Z_{\geq 0}^{k-m}\\1\leq i_k\leq\dots\leq i_{m+1}\leq m}} e^{\chi(\mathbf{i})}+\sum_{\substack{\mathbf{j}=(i_{k-1},\dots,i_{m+1})\in\Z_{\geq 0}^{k-1-m}\\1\leq i_{k-1}\leq\dots\leq i_{m+1}\leq m}} e^{\chi(\mathbf{j})}\right] & k > m + 1 \\
      \displaystyle (-1)^{k+m+1}e^{q_m}\left[\sum_{\substack{\mathbf{i}=(i_k,\dots,i_{m+1})\in\Z_{\geq 0}^{k-m}\\1\leq i_k\leq\dots\leq i_{m+1}\leq m}} e^{\chi(\mathbf{i})}+1\right] & k = m+1 \\
      \displaystyle 1-e^{q_m} & k=m,
\end{cases}
\end{equation}
where $$q_m:=\left\lceil \frac{m}{2}\right\rceil^2\alpha_0+\left(\left\lfloor\frac{m}{2}\right\rfloor^2+\left\lfloor\frac{m}{2}\right\rfloor\right)\alpha_1,$$ and $$\chi(i_\ell,\dots,i_{m+1}):={\sum_{\substack{j\text{ odd}\\ \ell\geq j \geq m+1}}i_j\alpha_0 + \sum_{\substack{j\text{ even} \\ \ell\geq j \geq m+1}}i_j\alpha_1}.$$

\subsection{Inductive formula for $d^k_{n,m}$}

In this subsection, we derive an inductive formula for the structure constants $d^k_{n,m}$.   First, we apply Theorem \ref{locthm} to compute the structure constants $d^{m}_{n,m}$ where $n\leq m$.

Let $w_m=s_{i_1}s_{i_2}\dots s_{i_m}$ be a reduced decomposition.  For $\ell\leq m$, define $\beta_\ell:=s_{i_1}s_{i_2}\dots s_{i_{\ell-1}}\alpha_{i_\ell}$
Then we have
\[
\beta_{\ell}=
\begin{cases}
\ell\alpha_0+(\ell-1)\alpha_1 & m\text{ odd}\\
(\ell-1)\alpha_0+\ell\alpha_1 & m\text{ even}.
\end{cases}
\]
Then combining Theorem \ref{locthm} with (\ref{locdvuv}) gives, for $n\leq m$,
\begin{equation}\label{dmnm}
d^m_{n,m}=\sum(-1)^{n}(e^{\beta_{j_1}}-1)\dots(e^{\beta_{j_p}}-1),
\end{equation} where the summation runs over all $1\leq j_1<\dots<j_p\leq m$ such that $w_n=s_{i_{j_1}}*\dots*s_{i_{j_p}}$, and the operation $*$ is defined as in Definition \ref{defn*}.

Now comparing $\hat{\OO}^1\cdot(\hat{\OO}^{n}\cdot \hat{\OO}^m)$ with $(\hat{\OO}^1\cdot\hat{\OO}^{n})\cdot \hat{\OO}^m$ in $K^0_T(\bar{\mathcal{X}})$, we obtain, for any $k$, 
\begin{equation}\label{ell}
\sum_i d^i_{n,m}d^k_{1,i}=\sum_jd^j_{1,n}d^k_{j,m}.
\end{equation}
Let $k> n$.  Then using that $d^j_{n,m}=0$ unless $j\geq \max\{n,m\}$, and solving for $d^{k}_{n,m}$ in (\ref{ell}), we obtain the following inductive relation for the structure constants:

\begin{prop}\label{inductprop}
For $k>n$, the $T$-equivariant structure constants in the basis $\hat{\OO}^k$ satisfy:
\begin{equation}
\label{induct}
d^{k}_{n,m}=\frac{1}{\left(d^{n}_{1,n}-d^{k}_{1,k}\right)}\left[\sum_{i=\max\{n,m\}}^{k-1}d^i_{n,m}d^{k}_{1,i} - \sum_{j=n+1}^{k}d^j_{1,n}d^{k}_{j,m}\right].
\end{equation}
\end{prop}

Note that all expressions on the right side of (\ref{induct}) may be assumed to be known by inducting on $k-n$.  The base cases are given by (\ref{dk1m}) and (\ref{dmnm}).  

\section{Structure constants for $\widehat{SL_2}$ in ordinary $K$-theory}

Let $K^0(\bar{\mathcal{X}})$ denote the Grothendieck group of coherent sheaves on $\bar{\mathcal{X}}$.  Denote by $\OO^k:=\overline{[\OO_{X_P^{w_k}}]}$, where as earlier, $P$ is the standard maximal parabolic subgroup.  Then, we have $$K^0(\bar{\mathcal{X}})=\prod_{k\in \Z_{\geq 0}}\Z \OO^k.$$  
Further, by (\ref{ev1}), we have, in $K^0(\bar{\mathcal{X}})$, $$\OO^n\cdot\OO^m=\sum_{k\geq n+m} d^k_{n,m}(1)\OO^k.$$

Consider the set $p(\ell,m)$.  By the results from the preceding section, the number of paths $p\in p(\ell,m)$ is the same as the number of $(\ell-m)$-tuples $(i_\ell,i_{\ell-1},\dots,i_{\ell-m+1})\in \Z_{\geq 0}^{\ell-m}$ satisfying $1\leq i_{\ell}\leq i_{\ell-1}\leq \dots \leq i_{\ell-m+1}\leq m$.  Hence the set $p(\ell,m)$ has cardinality $\binom{\ell-1}{m-1}$.
Thus, we have, by (\ref{Ddisjoint}),
\begin{equation}\label{Dcard}
\left|\mathcal{D}^{\Lambda_0}_{w_k,w_m}\right|=\binom{k-1}{m-1}+\binom{k-2}{m-1}.
\end{equation}

Evaluating (\ref{dk1m}) at $1$, we have, for $k>m$, 
\begin{equation}\label{d1ord}
d^k_{1,m}(1)=(-1)^{k+m+1}\left[\binom{k-1}{m-1}+\binom{k-2}{m-1}\right].
\end{equation}

Comparing $\OO^1\cdot(\OO^{n-1}\cdot \OO^m)$ with $(\OO^1\cdot\OO^{n-1})\cdot \OO^m$, we obtain, for any $\ell\geq 0$, 
\begin{equation}\label{ellord}
\sum_i d^i_{n-1,m}(1)d^\ell_{1,i}(1)=\sum_jd^j_{1,n-1}(1)d^\ell_{j,m}(1).
\end{equation}
Setting $\ell=n+m+k$ in (\ref{ellord}), where $k\geq 0$, using that $d^j_{n,m}(1)=0$ unless $j\geq n+m$, and solving for $d^{n+m+k}_{n,m}(1)$, we obtain the following inductive relation for the structure constants:

\begin{equation}
\label{inductord}
d^{n+m+k}_{n,m}(1)=\frac{1}{n}\left[\sum_{i=n+m-1}^{n+m+k-1}d^i_{n-1,m}(1)d^{n+m+k}_{1,i}(1) - \sum_{j=n+1}^{n+k}d^j_{1,n-1}(1)d^{n+m+k}_{j,m}(1)\right].
\end{equation}
In particular, choosing $k=0$ in (\ref{inductord}), and using (\ref{d1ord}) (with $k=m+1$),  we derive that
\begin{equation}\label{dnmord}
d^{n+m}_{n,m}(1)=\binom{n+m}{n}.
\end{equation}

Note that all expressions on the right side of (\ref{inductord}) may be assumed to be known by inducting on $n$ and $k$ simultaneously.  The base cases are covered by (\ref{d1ord}) and (\ref{dnmord}).  Thus (\ref{inductord}) completely determines the structure constants.  

\begin{thm}\label{dknmordthm}
The structure constants in the basis $\OO^k$ in ordinary $K$-theory are given by 
\begin{equation}\label{dknmord}
d^{n+m+k}_{n,m}(1)=(-1)^k\cdot\frac{(n+m+k-1)!}{(n-1)!(m-1)!k!}\cdot\frac{n+m+2k}{(n+k)(m+k)}.
\end{equation}
\end{thm}

\begin{proof}
It suffices to show that the above formula (\ref{dknmord}) satisfies the inductive relation (\ref{inductord}) and agrees with the known formulas for the base cases $d^k_{1,m}(1)$ (\ref{d1ord}) and $d^{n+m}_{n,m}(1)$ (\ref{dnmord}).

The base case $d^{n+m}_{n,m}(1)$ is immediately verified by letting $k=0$ in (\ref{dknmord}).  For the base case $d^k_{1,m}(1)$, let $\ell=k-1-m$.  Then, according to (\ref{dknmord}), 
\begin{align*}
d^k_{1,m}(1)=d^{1+m+\ell}_{1,m}(1)&=(-1)^\ell \cdot\frac{(m+\ell)!}{(m-1)!\ell!}\cdot\frac{(1+m+2\ell)}{(1+\ell)(m+\ell)}\\
&=(-1)^{k+1+m}\cdot\frac{(k-1)!}{(m-1)!(k-1-m)!}\cdot\frac{2k-1-m}{(k-m)(k-1)}.
\end{align*}
It is straightforward to verify that this is equal to $(-1)^{k+1+m}\left[\binom{k-1}{m-1}+\binom{k-2}{m-1}\right]$, as desired.

Now we must verify the inductive relation (\ref{inductord}).  
We first rewrite (\ref{inductord}) as 
\begin{equation}\label{inductord2}
\sum_{i=n+m-1}^{n+m+k-1}d^i_{n-1,m}(1)d^{n+m+k}_{1,i}(1) - \sum_{j=n}^{n+k}d^j_{1,n-1}(1)d^{n+m+k}_{j,m}(1)=0.
\end{equation}
Now according to (\ref{dknmord}), we have
\begin{align}
d^i_{n-1,m}(1)&=(-1)^{i+n+m+1}\cdot\frac{(i-1)!}{(n-2)!(m-1)!(i-n-m+1)!}\cdot\frac{(2i+1-n-m)}{(i-m)(i-n+1)},\label{d1}\\
d^{n+m+k}_{1,i}(1)&=(-1)^{n+m+k+i+1}\cdot\frac{(n+m+k-1)!}{(i-1)!(n+m+k-i-1)!}\cdot\frac{(2n+2m+2k-1-i)}{(n+m+k-i)(n+m+k-1)},\label{d2}\\
d^j_{1,n-1}(1)&=(-1)^{j+n}\cdot\frac{(j-1)!}{(n-2)!(j-n)!}\cdot\frac{(2j-n)}{(j-n+1)(j-1)},\label{d3}\\
d^{n+m+k}_{j,m}(1)&=(-1)^{n+k+j}\cdot\frac{(n+m+k-1)!}{(j-1)!(m-1)!(n+k-j)!}\cdot\frac{(2k+2n-j+m)}{(n+k)(n+m+k-j)}.\label{d4}
\end{align}

Then, plugging in the above formulas (\ref{d1})-(\ref{d4}), and letting $s=i-n-m+1$ and $t=j-n$, the left hand side of (\ref{inductord2}) becomes $$\sum_{s=0}^{k}(-1)^k\cdot\frac{(2s+n+m-1)(n+m+2k-s)}{(s+n-1)(s+m)(k+1-s)(n+m+k-1)}\cdot\frac{(n+m+k-1)!}{(n-2)!(m-1)!s!(k-s)!}$$ $$- \sum_{t=0}^{k}(-1)^k\cdot\frac{(2t+n)(2k+n+m-t)}{(t+n-1)(t+1)(n+k)(m+k-t)}\cdot\frac{(n+m+k-1)!}{(n-2)!(m-1)!t!(k-t)!}.$$
As we wish to show that this expression is equal to $0$, combing the two sums and dividing by the factor $(-1)^k\cdot\frac{(n+m+k-1)!}{(n-2)!(m-1)!}$ gives 
\begin{equation}\label{lhs}
\sum_{s=0}^{k}\frac{(n+m+2k-s)}{(s+n-1)s!(k-s)!}\left[\frac{2s+n+m-1}{(s+m)(k+1-s)(n+m+k-1)}-\frac{2s-n}{(s+1)(n+k)(m+k-s)}\right].
\end{equation}

We compute that the difference of the two fractions in the brackets in (\ref{lhs}) simplifies to $$\frac{(m-1)(2s-k)(s+n-1)(n+m+k+s)}{(s+m)(k+1-s)(k+m+n-1)(s+1)(n+k)(m+k-s)}.$$
Thus, expression (\ref{lhs}) being equal to zero is equivalent to the equation
\begin{equation}\label{sums=0}
\sum_{s=0}^{k}\frac{1}{s!(k-s)!}\left[\frac{(n+m+2k-s)(2s-k)(n+m+k+s)}{(s+m)(k+1-s)(s+1)(m+k-s)}\right]=0.
\end{equation}

Now letting $$f(s,k):=\frac{1}{s!(k-s)!}\left[\frac{(n+m+2k-s)(2s-k)(n+m+k+s)}{(s+m)(k+1-s)(s+1)(m+k-s)}\right],$$ we see that for $0\leq s \leq \frac{k}{2}$, we have $f(s,k)=-f(k-s,k)$.  This immediately implies (\ref{sums=0}), which in turn verifies (\ref{inductord2}), completing the proof.   

\end{proof}

Now let $\xi^k:=\overline{[\OO_{X_P^{w_k}}(-\partial X_P^{w_k})]}$ denote the class of $\OO_{X_P^{w_k}}(-\partial X_P^{w_k})$ in $K^0(\bar{\mathcal{X}})$.  Then we have, in $K^0(\bar{\mathcal{X}})$, $$\xi^n\cdot\xi^m=\sum_{k\geq n+m}b^k_{n,m}(1)\xi^k,$$ where $b^k_{n,m}$ are defined as in (\ref{scxi}).

From Theorem \ref{dknmordthm} one may compute that 
\begin{equation}\label{dddd}
d^j_{n,m}-d^j_{n+1,m}-d^j_{n,m+1}+d^j_{n+1,m+1}=(-1)^j\cdot\frac{(n+m+j-1)!}{n!m!j!}\cdot(n+m+2j).
\end{equation}
Then (\ref{dddd}) and (\ref{bd}) together with an inductive argument give the following corollary:

\begin{cor}\label{bcor}
The structure constants in the basis $\xi^k$ in ordinary $K$-theory are given by $$b^{n+m+k}_{n,m}(1)=(-1)^k\cdot\frac{(n+m+k)!}{n!m!k!}.$$
\end{cor}

\section{Structure constants for $\widehat{SL_2}$ in $T$-equivariant cohomology}

Let $G^{\min}:=\widehat{SL_2}$ be the minimal affine Kac-Moody group associated to $SL_2$.  Let $P$ be the standard maximal parabolic subgroup and let $\mathcal{X}:=G^{\min}/P$ denote the standard affine Grassmannian.  

Let $\{\e_i\}_{i=0}^\infty$ denote the Schubert basis in $T$-equivariant cohomology of $\mathcal{X}$.  Here we use the notation $\e_i:=\e_{w_i}$, where $\e_{w_i}$ is defined as in \cite[Theorem 11.3.9]{Kbook}.  Let $\mathbb{Z}[\alpha_0,\alpha_1]$ denote the graded ring of polynomials with integral coefficients in the simple roots $\alpha_0,\alpha_1$.  Further, let $\mathbb{Z}[\alpha_0,\alpha_1](k)$ denote the $k$-th graded piece of $\mathbb{Z}[\alpha_0,\alpha_1]$.  

We define the $T$-equivariant cohomology structure constants $c^k_{n,m}\in \mathbb{Z}[\alpha_0,\alpha_1](n+m-k)$ (see \cite[Corollary 11.3.17]{Kbook}) by
\begin{equation}\label{cknm}
\e_n\cdot\e_m=\sum_{k=\max\{n,m\}}^{n+m} c^k_{n,m}\e_k.
\end{equation}

By the Chevalley formula \cite[Theorem 11.17 (i)]{Kbook}, and using (\ref{womega}), we compute that
\begin{equation}\label{eqChevalley}
\e_1\cdot\e_m=q_m\e_m+(m+1)\e_{m+1},
\end{equation} where 
\begin{equation}\label{qm}
q_m:=\left\lceil \frac{m}{2}\right\rceil^2\alpha_0+\left(\left\lfloor\frac{m}{2}\right\rfloor^2+\left\lfloor\frac{m}{2}\right\rfloor\right)\alpha_1.
\end{equation}
In particular, $c^m_{1,m}=q_m$ and $c^{m+1}_{1,m}=m+1$.  

\begin{defn}\label{Qdef}
We let $Q^d_{i,j}$ denote the sum of all monomials of degree $d$ in the $j+1$ variables $q_i,q_{i+1},\dots,q_{i+j}$.
\end{defn}

From (\ref{eqChevalley}) and induction, we have 
\begin{equation}\label{eqpowers}
(\e_1)^n\cdot\e_m=\sum_{i=0}^n\frac{(m+i)!}{m!}Q^{n-i}_{m,i}\e_{m+i}.
\end{equation}
Further, we compute that 
\begin{equation}\label{eQe}
(\e_1)^n=\sum_{k=1}^n k!Q^{n-k}_{1,k-1} \e_k,
\end{equation}

Solving for $\e_n$ in equation (\ref{eqpowers}), we have $$\e_n=\frac{1}{n!}\left[(\e_1)^n-\sum_{k=1}^{n-1}k!Q^{n-k}_{1,k-1}\e_k\right].$$  
Assuming now that $n\leq m$, a computation yields $$\e_n\cdot\e_m=\sum_{i=0}^n \frac{1}{n!}\left[\frac{(m+i)!}{m!}Q^{n-i}_{m,i}-\sum_{k=1}^{n-1}k!Q^{n-k}_{1,k-1}c^{m+i}_{k,m}\right]\e_{m+i}.$$
Further, $c^j_{n,m}=0$ whenever $j<\max\{n,m\}$ or $j>n+m$.  Hence we have derived the following recursive formula for the structure constants:  

\begin{prop}\label{scindprop}
For $n\leq m$ and $0\leq i \leq n$,
\begin{equation}\label{scind}
c^{m+i}_{n,m} = \frac{1}{n!}\left[\frac{(m+i)!}{m!}Q^{n-i}_{m,i} - \sum_{k=\max\{i,1\}}^{n-1}k!Q^{n-k}_{1,k-1} c^{m+i}_{k,m}\right],
\end{equation}
where $Q^d_{i,j}$ is defined as in Definition \ref{Qdef}.
\end{prop}

Using the above formula (\ref{scind}), one may induct upwards on $n$ to compute the structure constants.  In addition, closed forms for the structure constants can be obtained inducting downwards on $i$.  

For example, it follows immediately, by letting $i=n$ in (\ref{scind}), that 
\begin{equation}\label{cn+mnm}
c^{n+m}_{n,m}=\binom{n+m}{n}.
\end{equation}
Now to compute a closed form for $c^{n+m-1}_{n,m}$, letting $i=n-1$ in (\ref{scind}) and using (\ref{cn+mnm}) gives
\begin{equation}\label{scind-1}
c^{n+m-1}_{n,m}=\frac{(n+m-1)!}{n!m!}\left[\sum_{j=1}^{n+m-1}q_j-\sum_{j=1}^{m-1}q_j-\sum_{j=1}^{n-1}q_j\right].
\end{equation}
A closed form for $\sum_{j=1}^k q_j$ is given by
\begin{equation}\label{sumq}
\sum_{j=1}^k q_j=
\begin{cases}
\frac{1}{12}k(k+1)(k+2)(\alpha_0+\alpha_1)& k\text{ even}\\
\frac{1}{12}(k+1)(3+2k+k^2)\alpha_0+\frac{1}{12}(k-1)(k+1)(k+3)\alpha_1&k\text{ odd}.
\end{cases}
\end{equation}
Thus, one may verify that (\ref{scind-1}) gives

\[
c^{n+m-1}_{n,m}=
\begin{cases}
\displaystyle\frac{1}{4}\cdot\frac{(n+m)!}{(n-1)!(m-1)!}(\alpha_0+\alpha_1) & n,m\text{ even}\vspace{.1in}\\
\displaystyle\frac{1}{4}\cdot\frac{(n+m)!}{n!m!}\left((1+nm)\alpha_0+(-1+nm)\alpha_1\right) & n,m\text{ odd}\vspace{.1in}\\
\displaystyle\frac{1}{4}\cdot\frac{(n+m-1)!}{(n-1)!m!}\left((-1+nm+m^2)\alpha_0+(1+nm+m^2)\alpha_1\right) & n\text{ even},m\text{ odd}\vspace{.1in}\\
\displaystyle\frac{1}{4}\cdot\frac{(n+m-1)!}{n!(m-1)!}\left((-1+nm+n^2)\alpha_0+(1+nm+n^2)\alpha_1\right) & n\text{ odd},m\text{ even}.
\end{cases}
\]

In general, to obtain a closed form for $c^{n+m-d}_{n,m}$ given closed forms for $c^{n+m}_{n,m},\dots,c^{n+m-d+1}_{n,m}$, one needs closed forms for $Q^1_{i,j},\dots,Q^d_{i,j}$.  It is easy to see that the following recurrence relation holds:
\begin{equation}\label{Qind}
Q^d_{i,j}=Q^{d-1}_{i,j}q_{i+j}+Q^d_{i,j-1}.
\end{equation}
By (\ref{sumq}) we obtain  
\[
Q^1_{i,j}=
\begin{cases}
\frac{1}{12}(3 i^2 + 2 j + 6 i j + 3 i^2 j + 3 j^2 + 3 i j^2 + j^3)\alpha_0\\+\frac{1}{12}(6i+3 i^2 + 2 j + 6 i j + 3 i^2 j + 3 j^2 + 3 i j^2 + j^3)\alpha_1&i,j\text{ even}\\\\
\frac{1}{12}(1+j)(3i+3i^2+2j+3ij+j^2)(\alpha_0+\alpha_1)& i,j\text{ odd}\\\\
\frac{1}{12}(1 + j) (3 + 3 i + 3 i^2 + 2 j + 3 i j + j^2)\alpha_0\\+\frac{1}{12}(1 + j) (-3 + 3 i + 3 i^2 + 2 j + 3 i j + j^2)\alpha_1& i\text{ even},j\text{ odd}\\\\
\frac{1}{12}(3 + 6 i + 3 i^2 + 5 j + 6 i j + 3 i^2 j + 3 j^2 + 3 i j^2 + j^3)\alpha_0\\+\frac{1}{12}(-3 + 3 i^2 - j + 6 i j + 3 i^2 j + 3 j^2 + 3 i j^2 + j^3)\alpha_1& i\text{ odd},j\text{ even}.
\end{cases}
\]

Now using the above and (\ref{Qind}) we obtain a closed form for $Q^2_{i,j}$, although we do not provide it here for the sake of space.  Then, using (\ref{scind}) we derive the following closed form for the structure constants $c^{n+m-2}_{n,m}$:

$$c^{n+m-2}_{n,m}=a^{n+m-2}_{n,m}\left(b^{n+m-2}_{n,m}(0)\alpha_0^2+b^{n+m-2}_{n,m}(1)\alpha_0\alpha_1+b^{n+m-2}_{n,m}(2)\alpha_1^2\right),$$
where
\[
a^{n+m-2}_{n,m}=
\begin{cases}
\displaystyle\frac{1}{8}\cdot\frac{(n+m)!}{(n-1)!(m-1)!}\cdot\frac{1}{n+m-1} & n,m\text{ even}\vspace{.1in}\\
\displaystyle\frac{1}{8}\cdot\frac{(n+m)!}{n!m!}\cdot\frac{(n-1)(m-1)}{n+m-1} & n,m\text{ odd}\vspace{.1in}\\
\displaystyle\frac{1}{8}\cdot\frac{(n+m-1)!}{n!(m-1)!}\cdot(n-1) & n\text{ even},m\text{ odd}\vspace{.1in}\\
\displaystyle\frac{1}{8}\cdot\frac{(n+m-1)!}{(n-1)!m!}\cdot(m-1) & n\text{ odd},m\text{ even}.
\end{cases}
\]

\[
b^{n+m-2}_{n,m}(0)=
\begin{cases}
\displaystyle\frac{1}{4}\left(nm^2+n^2m-n^2-m^2-3nm+4\right) & n,m\text{ even}\vspace{.1in}\\
\displaystyle\frac{1}{4}\left(nm^2+n^2m-nm-1\right) & n,m\text{ odd}\vspace{.1in}\\
\displaystyle\frac{1}{4}\left(nm^2+n^2m-n^2-nm+2n+3\right) & n\text{ even},m\text{ odd}\vspace{.1in}\\
\displaystyle\frac{1}{4}\left(nm^2+n^2m-m^2-nm+2m+3\right) & n\text{ odd},m\text{ even}.
\end{cases}
\]

\[
b^{n+m-2}_{n,m}(1)=
\begin{cases}
\displaystyle\frac{1}{2}\left(nm^2+n^2m-n^2-m^2-3nm+2n+2m-2\right) & n,m\text{ even}\vspace{.1in}\\
\displaystyle\frac{1}{2}\left(nm^2+n^2m-nm-1\right) & n,m\text{ odd}\vspace{.1in}\\
\displaystyle\frac{1}{2}\left(nm^2+n^2m-n^2-nm-1\right) & n\text{ even},m\text{ odd}\vspace{.1in}\\
\displaystyle\frac{1}{2}\left(nm^2+n^2m-m^2-nm-1\right) & n\text{ odd},m\text{ even}.
\end{cases}
\]

\[
b^{n+m-2}_{n,m}(2)=
\begin{cases}
\displaystyle\frac{1}{4}\left(nm^2+n^2m-n^2-m^2-3nm+4n+4m-4\right) & n,m\text{ even}\vspace{.1in}\\
\displaystyle\frac{1}{4}\left(nm^2+n^2m-n^2-nm+3\right) & n,m\text{ odd}\vspace{.1in}\\
\displaystyle\frac{1}{4}\left(nm^2+n^2m-n^2-nm-2n-1\right) & n\text{ even},m\text{ odd}\vspace{.1in}\\
\displaystyle\frac{1}{4}\left(nm^2+n^2m-m^2-nm-2m-1\right) & n\text{ odd},m\text{ even}.
\end{cases}
\]

Lastly, from \cite[Lemma 11.1.10 and Proposition 11.1.11 (1) and (3)]{Kbook} one may verify using induction that for $n\leq m$, a closed form for $c^m_{n,m}$ is given by
\[
c^{m}_{n,m}=
\begin{cases}
\displaystyle\binom{\frac{m+n}{2}}{n}\prod_{i=0}^{n-1}\left(\left(\frac{m-n}{2}+i\right)\alpha_0+\left(\frac{m-n}{2}+1+i\right)\alpha_1\right) & n,m\text{ even}\vspace{.1in}\\
\displaystyle\binom{\frac{m+n}{2}}{n}\prod_{i=0}^{n-1}\left(\left(\frac{m-n}{2}+1+i\right)\alpha_0+\left(\frac{m-n}{2}+i\right)\alpha_1\right) & n,m\text{ odd}\vspace{.1in}\\
\displaystyle\binom{\frac{m+n-1}{2}}{n}\prod_{i=0}^{n-1}\left(\left(\frac{m-n+1}{2}+i\right)\alpha_0+\left(\frac{m-n+1}{2}+1+i\right)\alpha_1\right) & n\text{ odd},m\text{ even}\vspace{.1in}\\
\displaystyle\binom{\frac{m+n-1}{2}}{n}\prod_{i=0}^{n-1}\left(\left(\frac{m-n+1}{2}+1+i\right)\alpha_0+\left(\frac{m-n+1}{2}+i\right)\alpha_1\right) & n\text{ even},m\text{ odd}.
\end{cases}
\]

\end{document}